\documentclass[11pt,reqno]{amsart}
\calclayout
\usepackage{amssymb,mathtools}
\usepackage[hidelinks]{hyperref}

\newtheorem{Thm}{Theorem}[section]
\newtheorem{Lem}[Thm]{Lemma}

\newtheorem{Cor}[Thm]{Corollary}
\theoremstyle{remark}
\newtheorem{Rem}[Thm]{Remark}

\newcommand{\CC}{\mathbb{C}}
\newcommand{\NN}{\mathbb{N}}

\begin{document}

\title[Double basic hypergeometric sums]{Double basic hypergeometric sums via a regularized Jackson $q$-integral}
\author{Takumi Maesaka}
\address{Faculty of Mathematics, Kyushu University, 744 Motooka, Nishi-ku, Fukuoka 819-0395, Japan}
\email{nozaki.takumi.912@s.kyushu-u.ac.jp}

\begin{abstract}
Using the Andrews--Askey integral, we derive a holomorphic extension formula for a Jackson $q$-integral. This formula allows identities established for terminating specializations to be continued to the nonterminating case. Our first application yields two companion binomial-type double-sum formulas. The second formula contains, as special cases, the nonterminating Sears--Carlitz transformation of Gasper and Rahman and Rahman's generating function for the Askey--Wilson polynomials. Applying the same method to the $q$-Chu--Vandermonde and Rogers' ${}_6\phi_5$ summations, we obtain two further integral representations for double sums of basic hypergeometric type. A balanced ${}_4\phi_3$ specialization of the latter identity is equivalent, via Sears' transformation, to the double-series transformation of Ismail, Rahman, and Suslov.
\end{abstract}

\maketitle

\section{Introduction}

Throughout this paper, let $0<|q|<1$.  For $a\in\CC$ and $n\in\NN\cup\{\infty\}$, we use the standard notation
\[
(a;q)_n=\prod_{j=0}^{n-1}(1-aq^j),
\qquad
(a_1,\ldots,a_r;q)_n=\prod_{i=1}^r(a_i;q)_n.
\]
We define the Jackson $q$-integral, without the conventional factor $1-q$, by
\begin{equation}\label{eq:q-integral}
\int_u^v g(t)\,d_qt
 \coloneqq \sum_{n=0}^{\infty}\bigl(vq^n g(vq^n)-uq^n g(uq^n)\bigr),
\end{equation}
whenever the series converges.

The Andrews--Askey integral \cite{aa} is
\begin{equation}\label{eq:AA-intro}
\int_u^v\frac{(tq/u,tq/v;q)_\infty}{(ct,dt;q)_\infty}\,d_qt
=v\frac{(q,u/v,vq/u,uvcd;q)_\infty}{(uc,ud,vc,vd;q)_\infty}.
\end{equation}
Various extensions of the Andrews--Askey integral and applications of terminating summations to $q$-beta integrals have been obtained by Al-Salam and Verma \cite{asv}, Fang \cite{fang}, Wang \cite{wang-aa,wang-pfaff}, and Liu \cite{liu-aa}.  The related passage from $q$-series expansions to beta-integral evaluations is also central to the work of Gasper and Schlosser \cite{gasper-schlosser}.  Comparing coefficients in a specialization of \eqref{eq:AA-intro} gives a family of $q$-moments closely related to an integral representation of the Al-Salam--Carlitz polynomials due to Wang \cite{w}.  Our first observation is that these moments give a holomorphic extension of
\[
\frac{(a/x;q)_\infty}{(q/x;q)_\infty}
\int_x^1\frac{(tq/x,tq;q)_\infty}{(at/x;q)_\infty}f(t)\,d_qt
\]
across $x=0$; see Theorem \ref{thm:regularity}.

This regularity result supplies the analytic step in a useful continuation argument.  For the binomial-type formulas, both sides are holomorphic in $t$ near $t=0$.  At the terminating specializations $c=q^{-N}$, finite identities show that the two sides agree as functions of $t$.  For each power of $t$, the corresponding coefficients on the two sides are polynomials of bounded degree in $c$.  Since these polynomials agree at the infinitely many points $c=q^{-N}$, they are identical.  This coefficientwise use of terminating specializations along a $q$-geometric sequence is a variant of Ismail's argument; a classical example is the proof of Bailey's ${}_6\psi_6$ summation by Askey and Ismail \cite{askey-ismail}.  The simplest instance yields the binomial-type formula in Theorem \ref{thm:binomial}, while Theorem \ref{thm:binomial-c} gives a companion formula involving the operator $\Delta_n^{(c)}$.  Specializations of the latter formula recover the nonterminating Sears--Carlitz transformation obtained by Gasper and Rahman \cite{gasper-rahman1986} and the generating function for the Askey--Wilson polynomials given by Rahman \cite{rahman1996}.  Applying analogous continuation arguments to the $q$-Chu--Vandermonde summation and Rogers' ${}_6\phi_5$ summation gives Theorems \ref{thm:qgauss} and \ref{thm:rogers}, respectively.

For reference, we use the standard definition
\[
{}_r\phi_s\!\left[
\begin{matrix}a_1,\ldots,a_r\\ b_1,\ldots,b_s\end{matrix};q,z
\right]
=\sum_{n=0}^{\infty}
\frac{(a_1,\ldots,a_r;q)_n}{(q,b_1,\ldots,b_s;q)_n}
\left((-1)^nq^{\binom n2}\right)^{1+s-r}z^n.
\]
As usual, terminating series are interpreted as finite sums, and the parameters
are assumed to be such that no denominator vanishes.
Throughout the paper, the double sums are understood as iterated sums, with
the terminating inner sum evaluated first.  In general, they are not
absolutely convergent as double series.

The remainder of the paper is organized as follows.  In Section \ref{sec:regularity}, we prove the regularity theorem and record a basic-hypergeometric form of it.  Section \ref{sec:binomial} gives two companion binomial-type double-sum identities and two classical specializations of the second.  In Sections \ref{sec:qgauss} and \ref{sec:rogers}, we derive the two remaining double-sum identities and their basic hypergeometric specializations.

\section{A regularized Jackson \texorpdfstring{$q$}{q}-integral and operator estimates}\label{sec:regularity}

We begin with the moment formula needed below.

\begin{Lem}\label{lem:moments}
For every nonnegative integer $m$, the following identity holds whenever
both sides are defined, and elsewhere by meromorphic continuation in the
parameters:
\begin{align}
&\int_x^1\frac{(tq/x,tq;q)_\infty}{(at/x;q)_\infty}t^m\,d_qt=\frac{(q,x,q/x;q)_\infty}{(a,a/x;q)_\infty}
\sum_{j=0}^m\frac{(q;q)_m(a;q)_{m-j}}{(q;q)_j(q;q)_{m-j}}x^j.
\label{eq:moment}
\end{align}
\end{Lem}

\begin{proof}
In \eqref{eq:AA-intro}, put $u=x$, $v=1$, and replace $(c,d)$ by $(a/x,z)$.  We obtain
\[
\int_x^1\frac{(tq/x,tq;q)_\infty}{(at/x,zt;q)_\infty}\,d_qt
=\frac{(q,x,q/x,az;q)_\infty}{(a,a/x,xz,z;q)_\infty}.
\]
The $q$-binomial theorem gives
\[
\frac{(az;q)_\infty}{(z,xz;q)_\infty}
=\sum_{m=0}^{\infty}z^m
\sum_{j=0}^m\frac{(a;q)_{m-j}}{(q;q)_j(q;q)_{m-j}}x^j.
\]
Comparing the coefficients of $z^m$ proves \eqref{eq:moment}.
\end{proof}

\begin{Thm}\label{thm:regularity}
Fix $a\in\CC$ such that $(a;q)_\infty\neq0$, and let
\[
f(z)=\sum_{n=0}^{\infty}c_nz^n
\]
be holomorphic in $|z|<R$, where $R>1$.  Wherever the following expression is initially defined, set
\[
\mathcal J_{a,f}(x)
\coloneqq \frac{(a/x;q)_\infty}{(q/x;q)_\infty}
\int_x^1\frac{(tq/x,tq;q)_\infty}{(at/x;q)_\infty}f(t)\,d_qt.
\]
Then $\mathcal J_{a,f}$ extends holomorphically to $|x|<R$.  More precisely,
\begin{equation}\label{eq:regularity-series}
\mathcal J_{a,f}(x)
=\frac{(q,x;q)_\infty}{(a;q)_\infty}
\sum_{k=0}^{\infty}x^k
\sum_{n=0}^{\infty}
\frac{(q^{k+1},a;q)_n}{(q;q)_n}c_{n+k}.
\end{equation}
The series on the right converges locally uniformly on $|x|<R$.  For fixed
$x$, the right-hand side also gives the meromorphic continuation in $a$,
with possible poles at $a\in\{q^{-j}:j\geq0\}$.
\end{Thm}

\begin{proof}
Multiplying \eqref{eq:moment} by $c_m$ and summing over $m$ gives
\begin{align*}
&\int_x^1\frac{(tq/x,tq;q)_\infty}{(at/x;q)_\infty}f(t)\,d_qt=\frac{(q,x,q/x;q)_\infty}{(a,a/x;q)_\infty}
\sum_{k=0}^{\infty}\frac{x^k}{(q;q)_k}
\sum_{n=0}^{\infty}
\frac{(q;q)_{n+k}(a;q)_n}{(q;q)_n}c_{n+k}.
\end{align*}
Since $(q;q)_{n+k}/(q;q)_k=(q^{k+1};q)_n$, this is \eqref{eq:regularity-series}.

It remains to justify the asserted convergence.  Fix $\rho$ with $1<\rho<R$.  Cauchy's estimate gives $|c_m|\leq M_\rho\rho^{-m}$.  The factors
\[
\frac{(q^{k+1},a;q)_n}{(q;q)_n}
\]
are bounded uniformly in $k,n\geq0$.  Consequently, for a constant $C_\rho$ independent of $k$,
\[
\left|\sum_{n=0}^{\infty}
\frac{(q^{k+1},a;q)_n}{(q;q)_n}c_{n+k}\right|
\leq C_\rho\rho^{-k}.
\]
Thus the outer series converges locally uniformly for $|x|<\rho$.  Since $\rho<R$ is arbitrary, the result follows.
\end{proof}

The following form makes clear how Theorem \ref{thm:regularity} regularizes a sum of two nonterminating basic hypergeometric series.

\begin{Cor}\label{cor:regular-combination}
Let $R>1$, and let $a,A_1,\ldots,A_r,B_1,\ldots,B_s$ satisfy
\[
 (a,B_1,\ldots,B_s;q)_\infty\neq0,
 \qquad |A_i|R<1\quad(1\leq i\leq r).
\]
Then
\begin{align}
&\sum_{k=0}^{\infty}
\frac{(A_1,\ldots,A_r,a/x;q)_k}
{(q,B_1,\ldots,B_s,q/x;q)_k}q^k \notag\\
&\quad+
\frac{(x,a/x,A_1,\ldots,A_r,B_1x,\ldots,B_sx;q)_\infty}
{(1/x,a,A_1x,\ldots,A_rx,B_1,\ldots,B_s;q)_\infty}\sum_{k=0}^{\infty}
\frac{(A_1x,\ldots,A_rx,a;q)_k}
{(q,B_1x,\ldots,B_sx,xq;q)_k}q^k
\label{eq:regular-combination}
\end{align}
extends holomorphically to $|x|<R$.  In particular, if $a\neq0$, $x=aq^N$,
and the remaining denominators do not vanish, its value is
\begin{equation}\label{eq:terminating-value}
\sum_{k=0}^N
\frac{(A_1,\ldots,A_r,q^{-N};q)_k}
{(q,B_1,\ldots,B_s,q^{1-N}/a;q)_k}q^k.
\end{equation}
\end{Cor}

\begin{proof}
Set
\[
h(z)\coloneqq \frac{(A_1,\ldots,A_r,B_1z,\ldots,B_sz;q)_\infty}
{(A_1z,\ldots,A_rz,B_1,\ldots,B_s;q)_\infty}.
\]
The assumptions imply that $h$ is holomorphic in $|z|<R$, and one has
\[
h(q^k)=\frac{(A_1,\ldots,A_r;q)_k}{(B_1,\ldots,B_s;q)_k}.
\]
Expanding the Jackson integral defining $\mathcal J_{a,h}(x)$ along its two $q$-lattices shows that the expression in \eqref{eq:regular-combination} is exactly $(q;q)_\infty^{-1}\mathcal J_{a,h}(x)$.  Theorem \ref{thm:regularity} proves the first assertion.  If $x=aq^N$, then $(a/x;q)_\infty=(q^{-N};q)_\infty=0$, so the second term in \eqref{eq:regular-combination} vanishes, while the first series terminates at $k=N$.  This gives \eqref{eq:terminating-value}.
\end{proof}

We conclude this section with estimates for two finite operators.  These
estimates will be used to justify convergence of the outer series and the
coefficientwise manipulations throughout the remainder of the paper.

\begin{Lem}\label{lem:operators}
Let $f(z)=\sum_{m\geq0}c_mz^m$ be holomorphic in $|z|<R$, where $R>1$, and define
\begin{align*}
\Delta_n(f)&\coloneqq \sum_{k=0}^n\frac{(q^{-n};q)_k}{(q;q)_k}f(q^k)q^k,\\
\Delta_n^{(a)}(f)&\coloneqq \sum_{k=0}^n\frac{(q^{-n},aq^n;q)_k}{(q;q)_k}f(q^k)q^k.
\end{align*}
For every $\rho$ with $1<\rho<R$ and every fixed $a\in\CC$, there is a
constant $C_{\rho,a}$ such that
\[
|\Delta_n(f)|,|\Delta_n^{(a)}(f)|\leq C_{\rho,a}\rho^{-n}
\qquad(n\geq0).
\]
\end{Lem}

\begin{proof}
It is enough to prove the estimate for $\Delta_n^{(a)}$ with arbitrary but
fixed $a$.  Indeed, $(0;q)_k=1$ gives
$\Delta_n^{(0)}(f)=\Delta_n(f)$, so the first estimate is the specialization
$a=0$ of the second estimate.

The finite $q$-binomial theorem gives
\[
(aq^n;q)_k
=\sum_{j=0}^k\frac{(q;q)_k}{(q;q)_j(q;q)_{k-j}}
(-a)^jq^{nj+\binom j2}.
\]
Consequently, for a monomial $z^m$, after interchanging the two finite sums and
putting $\ell=k-j$, we obtain
\begin{align*}
\Delta_n^{(a)}(z^m)
&=\sum_{k=0}^n\frac{(q^{-n};q)_k}{(q;q)_k}q^{(m+1)k}
\sum_{j=0}^k\frac{(q;q)_k(-a)^jq^{nj+\binom j2}}
{(q;q)_j(q;q)_{k-j}}\\
&=\sum_{j=0}^n
\frac{(-a)^jq^{nj+\binom j2}(q^{-n};q)_j}{(q;q)_j}q^{(m+1)j}
\sum_{\ell=0}^{n-j}
\frac{(q^{j-n};q)_\ell}{(q;q)_\ell}q^{(m+1)\ell}\\
&=\sum_{j=0}^n
\frac{(-a)^jq^{nj+\binom j2}(q^{-n};q)_j}{(q;q)_j}q^{(m+1)j}
(q^{m+1+j-n};q)_{n-j}\\
&=\sum_{j=\max(0,n-m)}^n
\frac{(-a)^jq^{nj+\binom j2}(q^{-n};q)_j}{(q;q)_j}q^{(m+1)j}
(q^{m+1+j-n};q)_{n-j}.
\end{align*}
Here the third line follows from the finite $q$-binomial theorem, and the
terms with $j<n-m$ vanish because
$(q^{m+1+j-n};q)_{n-j}$ then contains the factor $1-q^0$.  If $m\geq n$,
the finite sum defining $\Delta_n^{(a)}(z^m)$ is bounded uniformly in $m,n$
after using
\[
(q^{-n};q)_k=(-1)^kq^{-nk+\binom k2}(q^{n-k+1};q)_k.
\]
Indeed, since $m+1-n\geq1$ and all the remaining shifted factorials are
uniformly bounded,
\[
|\Delta_n^{(a)}(z^m)|
\leq C\sum_{k=0}^n |q|^{(m+1-n)k+\binom k2}
\leq C\sum_{k=0}^{\infty}|q|^{\binom k2}.
\]
If $m<n$, put $h=n-m$.  After substituting the displayed formula for
$(q^{-n};q)_j$, all $q$-shifted factorials that remain are uniformly
bounded, and hence
\[
|\Delta_n^{(a)}(z^m)|
\leq C\sum_{j=h}^n |a|^j|q|^{j(j+m)}.
\]
For all sufficiently large $n$, the last expression is at most
\[
C'\bigl(|a||q|^n\bigr)^h.
\]
Therefore, using $|c_m|\leq M_\rho\rho^{-m}$ and writing $m=n-h$, we obtain
\begin{align*}
\sum_{m=0}^{n-1}|c_m|\,|\Delta_n^{(a)}(z^m)|
&\leq C'M_\rho\rho^{-n}
\sum_{h=1}^n\bigl(\rho|a||q|^n\bigr)^h
=O(\rho^{-n}).
\end{align*}
The contribution from $m\geq n$ is bounded by
$C\sum_{m\geq n}|c_m|=O(\rho^{-n})$.  Enlarging the constant to cover the
finitely many remaining values of $n$ proves the assertion.
\end{proof}

\section{Binomial-type double sums}\label{sec:binomial}

We first record two finite identities that follow directly from the finite $q$-binomial theorem.  Coefficientwise continuation in the parameter $c$ then gives companion formulas involving $\Delta_n$ and $\Delta_n^{(c)}$.  The latter formula contains two classical identities as special cases.

\begin{Lem}\label{lem:binomial-finite}
Let $N$ be a nonnegative integer.  Then
\begin{align}
&\sum_{n=0}^N\frac{(q^{-N};q)_n}{(q;q)_n}t^n\Delta_n(f)
=(tq^{-N};q)_N\sum_{k=0}^N
\frac{(q^{-N};q)_k}{(q,q/t;q)_k}f(q^k)q^k,
\label{eq:binomial-finite}
\\
&\sum_{n=0}^N\frac{(q^{-N};q)_n}{(q;q)_n}t^n
\Delta_n^{(q^{-N})}(f)=(tq^{-N};q)_N\sum_{k=0}^N
\frac{(q^{-N};q)_{2k}}
{(q,tq^{-N},q/t;q)_k}f(q^k)q^k.
\label{eq:binomial-c-finite}
\end{align}
\end{Lem}

\begin{proof}
Interchange the two finite sums.  For fixed $k$, use
\[
(q^{-n};q)_k=(-1)^kq^{-nk+\binom k2}
\frac{(q;q)_n}{(q;q)_{n-k}}
\]
and put $n=k+m$.  The finite $q$-binomial theorem gives
\[
\sum_{m=0}^{N-k}\frac{(q^{k-N};q)_m}{(q;q)_m}
(tq^{-k})^m=(tq^{-N};q)_{N-k}.
\]
Consequently, the coefficient of $f(q^k)$ is
\[
\frac{(q;q)_N}{(q;q)_k(q;q)_{N-k}}
(tq^{-N})^k(tq^{-N};q)_{N-k}
=(tq^{-N};q)_N\frac{(q^{-N};q)_k}{(q,q/t;q)_k}q^k,
\]
which proves \eqref{eq:binomial-finite}.

For the second identity, the coefficient of $f(q^k)$ on the left is
\[
\frac{q^k}{(q;q)_k}\sum_{n=k}^N
\frac{(q^{-N};q)_n(q^{-n},q^{n-N};q)_k}{(q;q)_n}t^n.
\]
We use
\[
(q^{-N};q)_n(q^{n-N};q)_k=(q^{-N};q)_{n+k}
\]
and, after putting $n=k+m$,
\[
\frac{(q^{-k-m};q)_k}{(q;q)_{k+m}}
=\frac{(-1)^kq^{-km-\binom{k+1}{2}}}{(q;q)_m}.
\]
It follows that the coefficient equals
\begin{align*}
&(-1)^kq^{-\binom k2}\frac{t^k}{(q;q)_k}
\sum_{m=0}^{N-2k}
\frac{(q^{-N};q)_{2k+m}}{(q;q)_m}(tq^{-k})^m\\
&\quad=(-1)^kq^{-\binom k2}
\frac{t^k(q^{-N};q)_{2k}}{(q;q)_k}
\sum_{m=0}^{N-2k}
\frac{(q^{2k-N};q)_m}{(q;q)_m}(tq^{-k})^m\\
&\quad=(-1)^kq^{-\binom k2}
\frac{t^k(q^{-N};q)_{2k}}{(q;q)_k}
(tq^{k-N};q)_{N-2k}\\
&\quad=(tq^{-N};q)_N\frac{(q^{-N};q)_{2k}}{(q,tq^{-N},q/t;q)_k}q^k.
\end{align*}
This completes the proof.
\end{proof}

We shall also need the following specialization of the regularized integral at
the terminating points.  Let $N$ be a nonnegative integer.  If $a\neq0$,
$(a;q)_\infty\neq0$, $|aq^N|<R$, and the displayed denominators do not vanish,
then
\begin{equation}\label{eq:J-terminating}
\frac{1}{(q;q)_\infty}\mathcal J_{a,f}(aq^N)
=\sum_{k=0}^N\frac{(q^{-N};q)_k}
{(q,q^{1-N}/a;q)_k}f(q^k)q^k.
\end{equation}
To see this directly, expand the Jackson integral along its two $q$-lattices.
For generic $x$, the contribution from the upper lattice $s=q^k$ to
$(q;q)_\infty^{-1}\mathcal J_{a,f}(x)$ is
\[
\sum_{k=0}^\infty
\frac{(a/x;q)_k}{(q,q/x;q)_k}f(q^k)q^k.
\]
At $x=aq^N$, this sum terminates at $k=N$ and becomes the right-hand side
of \eqref{eq:J-terminating}.  The contribution from the lower lattice
$s=xq^k$ has the common factor $(a/x;q)_\infty$, and hence vanishes because
$(a/x;q)_\infty=(q^{-N};q)_\infty=0$.  This proves
\eqref{eq:J-terminating}.

\begin{Thm}\label{thm:binomial}
Let $f$ be holomorphic in $|z|<R$, where $R>1$, and suppose that $|t|<R$.
Then the identity
\begin{align}
&\sum_{n=0}^{\infty}\frac{(c;q)_n}{(q;q)_n}t^n\Delta_n(f)
=\frac{(c,ct;q)_\infty}{(q,t,q/t;q)_\infty}
\int_t^1\frac{(sq/t,sq;q)_\infty}{(cs;q)_\infty}f(s)\,d_qs
\label{eq:binomial-main}
\end{align}
holds wherever the displayed
expressions are defined. 
\end{Thm}

\begin{proof}
We begin under the restriction $|c|R<1$.  Put
\[
G_c(z)=\frac{(c;q)_\infty}{(cz;q)_\infty}f(z).
\]
Then $G_c$ is holomorphic in $|z|<R$, and the definition of the
regularized integral with $a=0$ gives
\begin{align}
&\frac{(c,ct;q)_\infty}{(q,t,q/t;q)_\infty}
\int_t^1\frac{(sq/t,sq;q)_\infty}{(cs;q)_\infty}f(s)\,d_qs
=\frac{(ct;q)_\infty}{(q,t;q)_\infty}\mathcal J_{0,G_c}(t).
\label{eq:binomial-a-zero}
\end{align}
Theorem \ref{thm:regularity} with $a=0$ shows that
$\mathcal J_{0,G_c}(t)$ is holomorphic for $|t|<R$.  Since the prefactor
$(ct;q)_\infty/(q,t;q)_\infty$ is holomorphic near $t=0$, the integral
side of \eqref{eq:binomial-main} is holomorphic there.  Lemma
\ref{lem:operators} similarly shows that the series side is holomorphic
for $|t|<R$.

We next make the dependence of the Taylor coefficients on $c$ explicit.
Write $f(z)=\sum_{r\geq0}f_rz^r$.  Lemma \ref{lem:moments}, with
$a=ct$ and $x=t$, gives
\[
\frac{(c,ct;q)_\infty}{(q,t,q/t;q)_\infty}
\int_t^1\frac{(sq/t,sq;q)_\infty}{(cs;q)_\infty}s^r\,d_qs
=\sum_{k=0}^r
\frac{(q;q)_r(ct;q)_{r-k}}{(q;q)_k(q;q)_{r-k}}t^k.
\]
Summing this identity against $f_r$ and putting $j=r-k$, with the
interchange justified by Cauchy's estimate for the coefficients of $f$,
yields the locally uniformly convergent expansion
\begin{align}
&\frac{(c,ct;q)_\infty}{(q,t,q/t;q)_\infty}
\int_t^1\frac{(sq/t,sq;q)_\infty}{(cs;q)_\infty}f(s)\,d_qs=\sum_{k=0}^{\infty}t^k\sum_{j=0}^{\infty}
\frac{(q^{k+1},ct;q)_j}{(q;q)_j}f_{j+k}.
\label{eq:binomial-regular-series}
\end{align}
The series on the right converges locally uniformly in $(c,t)$ on
$\CC\times\{|t|<R\}$ and hence also supplies the continuation beyond the
initial restriction $|c|R<1$.
Let $B_m(c)$ denote the coefficient of $t^m$ in this expression.  The
finite $q$-binomial theorem gives, for $0\leq k\leq m$,
\[
[t^{m-k}](ct;q)_j=
\begin{cases}
\displaystyle
\frac{(q;q)_j(-c)^{m-k}q^{\binom{m-k}{2}}}
{(q;q)_{m-k}(q;q)_{j-m+k}},&j\geq m-k,\\[3mm]
0,&j<m-k.
\end{cases}
\]
Consequently,
\begin{align}
B_m(c)
&=\sum_{k=0}^m
\frac{(-c)^{m-k}q^{\binom{m-k}{2}}}{(q;q)_{m-k}}
\sum_{j=m-k}^{\infty}
\frac{(q^{k+1};q)_j}{(q;q)_{j-m+k}}f_{j+k}.
\label{eq:binomial-coefficient}
\end{align}
The inner sums converge absolutely by Cauchy's estimate.  Formula
\eqref{eq:binomial-coefficient} therefore shows directly that $B_m(c)$
is a polynomial in $c$ of degree at most $m$.  On the other hand, the
coefficient on the series side is
\[
A_m(c)=\frac{(c;q)_m}{(q;q)_m}\Delta_m(f),
\]
which is also a polynomial in $c$ of degree at most $m$.

We next put $c=q^{-N}$.  Since $t=(tq^{-N})q^N$,
\eqref{eq:J-terminating} with $a=tq^{-N}$ gives
\begin{align*}
&\frac{(tq^{-N};q)_\infty}{(q,t;q)_\infty}
\mathcal J_{tq^{-N},f}(t)=(tq^{-N};q)_N\sum_{k=0}^N
\frac{(q^{-N};q)_k}{(q,q/t;q)_k}f(q^k)q^k.
\end{align*}
By \eqref{eq:binomial-finite}, this is precisely the left-hand side of
\eqref{eq:binomial-main} at $c=q^{-N}$.  Since
\eqref{eq:binomial-regular-series} supplies the holomorphic continuation
of the integral side near $t=0$, it follows that
$A_m(q^{-N})=B_m(q^{-N})$ for every $N\geq0$.  Since $A_m$ and
$B_m$ are polynomials and agree at the infinitely many distinct points
$q^{-N}$, they are identical.  Hence the two sides of
\eqref{eq:binomial-main} agree near $t=0$ in the initial parameter domain.
Meromorphic continuation in $c$, followed by analytic continuation in $t$
throughout the stated domain, proves the identity. 
\end{proof}

The companion formula below treats the finite operator with the additional parameter $cq^n$.

\begin{Thm}\label{thm:binomial-c}
Let $f$ be holomorphic in $|z|<R$, where $R>1$, and suppose that $|t|<R$.
Then the identity
\begin{align}
&\sum_{n=0}^{\infty}\frac{(c;q)_n}{(q;q)_n}t^n\Delta_n^{(c)}(f) 
=\frac{(c;q)_\infty}{(q,t,q/t;q)_\infty}
\int_t^1\frac{(sq/t,sq,cts;q)_\infty}{(cs^2;q)_\infty}f(s)\,d_qs
\label{eq:binomial-c-main}
\end{align}
holds wherever the displayed expressions are defined. 
\end{Thm}

\begin{proof}
We begin under the restriction $|c|R^2<1$.

We first justify holomorphy in $t$ near zero.  Choose $\rho$ with
$1<\rho<R$.  Since $(c;q)_n/(q;q)_n$ is bounded uniformly in $n$,
Lemma \ref{lem:operators} gives
\[
\left|
\frac{(c;q)_n}{(q;q)_n}\Delta_n^{(c)}(f)
\right|
\leq C_\rho\rho^{-n}.
\]
Consequently, the series on the left-hand side of
\eqref{eq:binomial-c-main} converges locally uniformly for $|t|<\rho$.  Since
$\rho$ may be chosen arbitrarily in $(1,R)$, the series defines a holomorphic
function for $|t|<R$.

For the right-hand side, introduce an auxiliary parameter $b$ and set
\[
I(b,t)\coloneqq
\int_t^1\frac{(sq/t,sq,bs;q)_\infty}{(cs^2;q)_\infty}
f(s)\,d_qs,
\qquad
\Phi(b,t)\coloneqq
\frac{(c;q)_\infty}{(q,t,q/t;q)_\infty}I(b,t).
\]
Under the initial restriction $|c|R^2<1$, the function
\[
F_b(s)=\frac{(bs;q)_\infty}{(cs^2;q)_\infty}f(s)
\]
is holomorphic in $|s|<R$ for every $b\in\CC$.  From the definition of the
regularized integral with $a=0$,
\[
\Phi(b,t)=\frac{(c;q)_\infty}{(q,t;q)_\infty}
\mathcal J_{0,F_b}(t).
\]
Choose $r_0$ with $0<r_0<1$.  Theorem \ref{thm:regularity} shows, for each
fixed $b$, that $t\mapsto\Phi(b,t)$ extends holomorphically to $|t|<r_0$.
Moreover, the series in that theorem gives
\[
\mathcal J_{0,F_b}(0)=(q;q)_\infty F_b(1),
\qquad
\Phi(b,0)=(b;q)_\infty f(1),
\]
so $b\mapsto\Phi(b,0)$ is entire.

We next prove holomorphy in $b$ directly from the definition of the Jackson
$q$-integral.  For $t\neq0$, expansion along the two $q$-lattices gives
\begin{align*}
I(b,t)
&=\sum_{n=0}^{\infty}q^n\left(
\frac{(q^{n+1}/t,q^{n+1},bq^n;q)_\infty}
{(cq^{2n};q)_\infty}f(q^n)-t\frac{(q^{n+1},tq^{n+1},btq^n;q)_\infty}
{(ct^2q^{2n};q)_\infty}f(tq^n)
\right).
\end{align*}
This series converges uniformly on compact subsets of
$\CC\times\{0<|t|<r_0\}$: the infinite products in the summands are
uniformly bounded, the denominator products are bounded away from zero,
the values of $f$ are uniformly bounded, and the factor $|q|^n$ gives a
convergent majorant.  Thus $I$ is jointly holomorphic there.  Let
\[
E=\{q^N:N\geq1\}.
\]
It follows that $\Phi$ is jointly holomorphic when
$0<|t|<r_0$ and $t\notin E$.  If $t_0=q^N$, then the two lattices in the
definition of $I(b,t_0)$ cancel from the $N$th term onward, and therefore
\begin{align*}
I(b,t_0)
&=\sum_{n=0}^{N-1}q^n
\frac{(q^{n+1-N},q^{n+1},bq^n;q)_\infty}
{(cq^{2n};q)_\infty}f(q^n)=0.
\end{align*}
Indeed, $(q^{n+1-N};q)_\infty=0$ for $0\leq n<N$. Since $I(b,t_0)=0$ for every $b$ and $(q/t;q)_\infty$ has a simple zero at $t=t_0$, the holomorphic division theorem shows that $I(b,t)/(q/t;q)_\infty$ extends jointly holomorphically across $t=t_0$. In particular,
$\Phi$ is jointly holomorphic across $t=t_0$.  Hence it is jointly
holomorphic for $0<|t|<r_0$.  Together with the holomorphic extension in
$t$ for each fixed $b$ and the formula for $\Phi(b,0)$, this shows that
$\Phi$ is separately holomorphic in $b$ and $t$ on every bidisk
$|b|<B$, $|t|<r_0$.  Hartogs' theorem implies joint holomorphy there.
Finally, the right-hand side of
\eqref{eq:binomial-c-main} is $\Phi(ct,t)$ and is therefore holomorphic at
$t=0$.

We next verify the required polynomial dependence on $c$.  Put
\begin{align*}
K_{c,t}(z)&=\frac{(cz,ctz;q)_\infty}{(ct,cz^2;q)_\infty},
&H_c(z)&=\frac{(cz;q)_\infty}{(cz^2;q)_\infty}f(z).
\end{align*}
The definition of the regularized integral gives, by cancellation of
infinite products,
\begin{align}
&\frac{(c;q)_\infty}{(q,t,q/t;q)_\infty}
\int_t^1\frac{(sq/t,sq,cts;q)_\infty}{(cs^2;q)_\infty}f(s)\,d_qs
=\frac{(ct;q)_\infty}{(q,t;q)_\infty}
\mathcal J_{ct,K_{c,t}f}(t).
\label{eq:binomial-c-regularized}
\end{align}

The $q$-binomial theorem gives
\[
\frac{(ctz;q)_\infty}{(ct;q)_\infty}
=\sum_{j=0}^{\infty}\frac{(z;q)_j}{(q;q)_j}(ct)^j,
\]
and hence
\[
K_{c,t}(z)f(z)
=\sum_{j=0}^{\infty}
\frac{c^j}{(q;q)_j}(z;q)_jH_c(z)t^j.
\]
Now let $B_m(c)$ denote the coefficient of $t^m$ in the right-hand side of
\eqref{eq:binomial-c-main}.  By the last two displays and linearity of the
regularized integral,
\[
B_m(c)=\sum_{j=0}^m\frac{c^j}{(q;q)_j}
[t^{m-j}]\frac{(ct;q)_\infty}{(q,t;q)_\infty}
\mathcal J_{ct,(z;q)_jH_c}(t).
\]
For $0\leq j\leq m$, Lemma \ref{lem:moments} gives
\begin{align}
[t^{m-j}]\frac{(ct;q)_\infty}{(q,t;q)_\infty}
\mathcal J_{ct,(z;q)_jH_c}(t)
&=\frac{(c;q)_{m-j}}{(q;q)_{m-j}}
\Delta_{m-j}\!\left((z;q)_jH_c(z)\right).
\label{eq:regularized-coefficient}
\end{align}
To verify this, put $n=m-j$.  For a monomial $z^r$, Lemma
\ref{lem:moments} gives
\[
\frac{(ct;q)_\infty}{(q,t;q)_\infty}\mathcal J_{ct,z^r}(t)
=\sum_{\ell=0}^r
\frac{(q;q)_r(ct;q)_{r-\ell}}
{(q;q)_\ell(q;q)_{r-\ell}}t^\ell.
\]
If $n>r$, the coefficient of $t^n$ in this expression vanishes, as does
\[
\Delta_n(z^r)=(q^{r+1-n};q)_n.
\]
If $0\leq n\leq r$, the finite $q$-binomial theorem gives
\begin{align*}
[t^n]\sum_{\ell=0}^r
\frac{(q;q)_r(ct;q)_{r-\ell}}
{(q;q)_\ell(q;q)_{r-\ell}}t^\ell&=\frac{(q;q)_r}{(q;q)_{r-n}}
\sum_{\ell=0}^n
\frac{(-c)^{n-\ell}q^{\binom{n-\ell}{2}}}
{(q;q)_\ell(q;q)_{n-\ell}}\\
&=\frac{(c;q)_n}{(q;q)_n}
\frac{(q;q)_r}{(q;q)_{r-n}}\\
&=\frac{(c;q)_n}{(q;q)_n}\Delta_n(z^r).
\end{align*}
Now write $(z;q)_jH_c(z)=\sum_{r\geq0}h_rz^r$.  Taylor expansion and
local uniform convergence allow us to sum the monomial identity against
$h_r$.  This gives \eqref{eq:regularized-coefficient}.

Substituting \eqref{eq:regularized-coefficient} into the expression for
$B_m(c)$, we obtain directly
\begin{align}
B_m(c)
&=\sum_{j=0}^m
\frac{c^j(c;q)_{m-j}}{(q;q)_j(q;q)_{m-j}}
\Delta_{m-j}\!\left((z;q)_jH_c(z)\right).
\label{eq:binomial-c-coefficient}
\end{align}
In the finite sum defining
$\Delta_{m-j}\!\left((z;q)_jH_c(z)\right)$, one has $0\leq k\leq m-j$
and $H_c(q^k)=(cq^k;q)_k f(q^k)$.  Therefore, this quantity is a polynomial
in $c$ of degree at most $m-j$.  The $j$th summand in
\eqref{eq:binomial-c-coefficient} therefore has degree at most
$j+(m-j)+(m-j)=2m-j\leq2m$.  Hence $B_m(c)$ is a polynomial in $c$ of
degree at most $2m$.  On the other hand, the coefficient
\[
A_m(c)=\frac{(c;q)_m}{(q;q)_m}\Delta_m^{(c)}(f)
\]
on the left-hand side is also a polynomial in $c$ of degree at most $2m$,
because the terms of $\Delta_m^{(c)}(f)$ contain $(cq^m;q)_k$ with
$0\leq k\leq m$.

It remains to compare these coefficients.  For generic $t$, expanding the
Jackson integral in \eqref{eq:binomial-c-main} along its two $q$-lattices
and then taking the meromorphic limit $c\to q^{-N}$ gives
\begin{align*}
&\lim_{c\to q^{-N}}
\frac{(c;q)_\infty}{(q,t,q/t;q)_\infty}
\int_t^1\frac{(sq/t,sq,cts;q)_\infty}
{(cs^2;q)_\infty}f(s)\,d_qs=(tq^{-N};q)_N\sum_{k=0}^N
\frac{(q^{-N};q)_{2k}}
{(q,tq^{-N},q/t;q)_k}f(q^k)q^k.
\end{align*}
In this limit, the lower $q$-lattice contribution vanishes because it contains
$(q^{-N};q)_\infty$, while the upper contribution terminates because
$(q^{-N};q)_{2k}=0$ for $2k>N$.  By
\eqref{eq:binomial-c-finite}, this is exactly the left-hand side of
\eqref{eq:binomial-c-main} at $c=q^{-N}$.  Hence
$A_m(q^{-N})=B_m(q^{-N})$ for every $N\geq0$.

Since $A_m$ and $B_m$ are polynomials and agree at the infinitely many
distinct points $q^{-N}$, they are identical.  Thus the two sides of
\eqref{eq:binomial-c-main} agree near $t=0$ in the initial parameter
domain.  Meromorphic continuation in $c$, followed by analytic continuation
in $t$ within $|t|<R$ away from the poles of the displayed expressions, proves the stated
identity.
\end{proof}
When $f$ is a quotient of infinite products, Theorem \ref{thm:binomial} yields an explicit transformation into a sum of two series.

\begin{Cor}\label{cor:binomial-hypergeometric}
Suppose that
\[
|A_i|\max\{1,|t|\}<1
\qquad(1\leq i\leq r).
\]
Then
\begin{align}
&\sum_{n=0}^{\infty}\frac{(c;q)_n}{(q;q)_n}t^n
\sum_{k=0}^n\frac{(q^{-n},A_1,\ldots,A_r;q)_k}
{(q,B_1,\ldots,B_s;q)_k}q^k \notag\\
&=\frac{(ct;q)_\infty}{(t;q)_\infty}
\sum_{k=0}^{\infty}\frac{(c,A_1,\ldots,A_r;q)_k}
{(q,q/t,B_1,\ldots,B_s;q)_k}q^k \notag\\
&\quad+\frac{(c,A_1,\ldots,A_r,B_1t,\ldots,B_st;q)_\infty}
{(1/t,B_1,\ldots,B_s,A_1t,\ldots,A_rt;q)_\infty}
\sum_{k=0}^{\infty}\frac{(ct,A_1t,\ldots,A_rt;q)_k}
{(q,tq,B_1t,\ldots,B_st;q)_k}q^k.
\label{eq:binomial-two-series}
\end{align}
\end{Cor}

\begin{proof}
For generic values of the denominator parameters, define
\[
f(z)=\frac{(A_1,\ldots,A_r,B_1z,\ldots,B_sz;q)_\infty}
{(A_1z,\ldots,A_rz,B_1,\ldots,B_s;q)_\infty}.
\]
The hypothesis permits us to choose $R>\max\{1,|t|\}$ such that
$|A_i|R<1$ for $1\leq i\leq r$.  Hence $f$ is holomorphic in $|z|<R$,
and
\[
f(q^k)=\frac{(A_1,\ldots,A_r;q)_k}
{(B_1,\ldots,B_s;q)_k}.
\]
The result now follows from Theorem \ref{thm:binomial} by expanding the
Jackson integral along its two $q$-lattices.  The remaining admissible
parameter values follow by meromorphic continuation.
\end{proof}

The following explicit form of Theorem \ref{thm:binomial-c} yields the two classical formulas below.

\begin{Cor}\label{cor:binomial-c-hypergeometric}
Suppose that
\[
|A_i|\max\{1,|t|\}<1
\qquad(1\leq i\leq r).
\]
Then
\begin{align}
&\sum_{n=0}^{\infty}\frac{(c;q)_n}{(q;q)_n}t^n
\sum_{k=0}^n
\frac{(q^{-n},cq^n,A_1,\ldots,A_r;q)_k}
{(q,B_1,\ldots,B_s;q)_k}q^k\notag\\
&=\frac{(ct;q)_\infty}{(t;q)_\infty}
\sum_{k=0}^{\infty}
\frac{(\sqrt c,-\sqrt c,\sqrt{cq},-\sqrt{cq},
A_1,\ldots,A_r;q)_k}
{(q,q/t,ct,B_1,\ldots,B_s;q)_k}q^k\notag\\
&\quad+
\frac{(\sqrt c,-\sqrt c,\sqrt{cq},-\sqrt{cq},
A_1,\ldots,A_r,ct^2,B_1t,\ldots,B_st;q)_\infty}
{(1/t,B_1,\ldots,B_s,t\sqrt c,-t\sqrt c,
t\sqrt{cq},-t\sqrt{cq},A_1t,\ldots,A_rt;q)_\infty}\notag\\
&\qquad\cdot
\sum_{k=0}^{\infty}
\frac{(t\sqrt c,-t\sqrt c,t\sqrt{cq},-t\sqrt{cq},
A_1t,\ldots,A_rt;q)_k}
{(q,tq,ct^2,B_1t,\ldots,B_st;q)_k}q^k.
\label{eq:binomial-c-two-series}
\end{align}
\end{Cor}

\begin{proof}
For generic values of the denominator parameters, define
\[
f(z)=\frac{(A_1,\ldots,A_r,B_1z,\ldots,B_sz;q)_\infty}
{(A_1z,\ldots,A_rz,B_1,\ldots,B_s;q)_\infty}.
\]
As in the proof of Corollary \ref{cor:binomial-hypergeometric}, the
hypothesis permits us to choose $R>\max\{1,|t|\}$ such that $f$ is
holomorphic in $|z|<R$, and
\[
f(q^k)=\frac{(A_1,\ldots,A_r;q)_k}
{(B_1,\ldots,B_s;q)_k}.
\]
Thus the left-hand side agrees with that of Theorem
\ref{thm:binomial-c}.  Expanding the Jackson integral in
\eqref{eq:binomial-c-main} along its two $q$-lattices gives
\eqref{eq:binomial-c-two-series}.  The remaining admissible parameter
values follow by meromorphic continuation.
\end{proof}

\begin{Rem}[The nonterminating Sears--Carlitz transformation]\label{rem:sears-carlitz}
In Corollary \ref{cor:binomial-c-hypergeometric}, take $(c,t)=(a,x)$ and
\[
A_1=\frac{aq}{bc},
\qquad
(B_1,B_2)=\left(\frac{aq}{b},\frac{aq}{c}\right).
\]
The terminating $q$-Pfaff--Saalsch\"utz summation \cite[Eq.~(II.12)]{gr} gives
\[
\sum_{k=0}^n\frac{(q^{-n},aq^n,aq/bc;q)_k}
{(q,aq/b,aq/c;q)_k}q^k
=\frac{(b,c;q)_n}{(aq/b,aq/c;q)_n}
\left(\frac{aq}{bc}\right)^n.
\]
Thus \eqref{eq:binomial-c-two-series} becomes
\begin{align}
&{}_3\phi_2\!\left[
\begin{matrix}a,b,c\\ aq/b,aq/c\end{matrix};q,\frac{aqx}{bc}\right]\\
&=\frac{(ax;q)_\infty}{(x;q)_\infty}
{}_5\phi_4\!\left[
\begin{matrix}\sqrt a,-\sqrt a,\sqrt{aq},-\sqrt{aq},aq/bc\\
aq/b,aq/c,ax,q/x\end{matrix};q,q\right] \notag\\
&\quad+\frac{(a,aq/bc,aqx/b,aqx/c;q)_\infty}
{(aq/b,aq/c,aqx/bc,1/x;q)_\infty}
{}_5\phi_4\!\left[
\begin{matrix}x\sqrt a,-x\sqrt a,x\sqrt{aq},-x\sqrt{aq},aqx/bc\\
aqx/b,aqx/c,xq,ax^2\end{matrix};q,q\right].
\label{eq:nonterminating-sears-carlitz}
\end{align}
This is the nonterminating Sears--Carlitz transformation of Gasper and Rahman \cite{gasper-rahman1986}; see also \cite[Eq.~(3.4.1)]{gr}.  The identity is initially valid when $|aqx/bc|<1$ and no denominator vanishes, and it extends elsewhere by analytic continuation.
\end{Rem}

\begin{Rem}[Rahman's Askey--Wilson generating function]\label{rem:rahman-aw}
Let $x=\cos\theta$.  The Askey--Wilson polynomials are defined by
\[
p_n(x;a,b,c,d\mid q)
=a^{-n}(ab,ac,ad;q)_n
{}_4\phi_3\!\left[
\begin{matrix}q^{-n},abcdq^{n-1},ae^{i\theta},ae^{-i\theta}\\
ab,ac,ad\end{matrix};q,q\right].
\]
Write
\[
r_n(x;a,b,c,d\mid q)
={}_4\phi_3\!\left[
\begin{matrix}q^{-n},abcdq^{n-1},ae^{i\theta},ae^{-i\theta}\\
ab,ac,ad\end{matrix};q,q\right].
\]
In Corollary \ref{cor:binomial-c-hypergeometric}, replace the parameter
$c$ appearing there by $abcd/q$ and set
\begin{align*}
(A_1,A_2)&=\left(ae^{i\theta},ae^{-i\theta}\right),\\
(B_1,B_2,B_3)&=(ab,ac,ad).
\end{align*}
After substitution and cancellation of common shifted factorials,
\eqref{eq:binomial-c-two-series} gives Rahman's generating function
\begin{align}
&\sum_{n=0}^{\infty}
\frac{(abcd/q;q)_n}{(q;q)_n}t^n r_n(x;a,b,c,d\mid q) \notag\\
&=\frac{(abcdt/q;q)_\infty}{(t;q)_\infty}
{}_6\phi_5\!\left[
\begin{matrix}
\sqrt{abcd/q},-\sqrt{abcd/q},\sqrt{abcd},-\sqrt{abcd},
ae^{i\theta},ae^{-i\theta}\\
ab,ac,ad,abcdt/q,q/t
\end{matrix};q,q\right] \notag\\
&\quad+\frac{(abcd/q,abt,act,adt,ae^{i\theta},ae^{-i\theta};q)_\infty}
{(ab,ac,ad,ate^{i\theta},ate^{-i\theta},1/t;q)_\infty}\notag\\
&\qquad\cdot{}_6\phi_5\!\left[
\begin{matrix}
t\sqrt{abcd/q},-t\sqrt{abcd/q},t\sqrt{abcd},-t\sqrt{abcd},
ate^{i\theta},ate^{-i\theta}\\
abt,act,adt,abcdt^2/q,tq
\end{matrix};q,q\right].
\label{eq:rahman-aw-generating}
\end{align}
This is Rahman's generating function \cite[Eq.~(4.9)]{rahman1996}.  It is
initially valid, for example, when $|a|\max\{1,|t|\}<1$ and no denominator
vanishes; it extends elsewhere by analytic continuation.
\end{Rem}

\section{A double sum of \texorpdfstring{$q$}{q}-Gauss type}\label{sec:qgauss}

The combination of the $q$-Chu--Vandermonde summation with the Andrews--Askey integral has led to several extensions and integral identities; see, for example, Fang \cite{fang} and Wang \cite{wang-aa,wang-pfaff}.  Here, the additional feature is that the terminating transformation is formulated for an arbitrary holomorphic function $f$, and the regularized integral of Section \ref{sec:regularity} then supplies the continuation step.

We first establish a terminating identity.

\begin{Lem}\label{lem:qgauss-finite}
Let $N$ be a nonnegative integer.  Then
\begin{align}
&\sum_{n=0}^N
\frac{(a,q^{-N};q)_n}{(c,q;q)_n}
\left(\frac{cq^N}{a}\right)^n\Delta_n(f)=\frac{(c/a;q)_N}{(c;q)_N}
\sum_{k=0}^N
\frac{(a,q^{-N};q)_k}{(q,aq^{1-N}/c;q)_k}f(q^k)q^k.
\label{eq:qgauss-finite}
\end{align}
\end{Lem}

\begin{proof}
After interchanging the two finite sums, we need only evaluate
\[
\sum_{n=k}^N
\frac{(a,q^{-N};q)_n(q^{-n};q)_k}{(c,q;q)_n}
\left(\frac{cq^N}{a}\right)^n.
\]
Using
\[
(q^{-n};q)_k=(-1)^kq^{\binom k2-nk}
\frac{(q;q)_n}{(q;q)_{n-k}}
\]
and then replacing $n$ by $n+k$, we obtain a terminating ${}_2\phi_1$.
The $q$-Chu--Vandermonde summation \cite[Eq.~(II.7)]{gr} gives
\begin{align*}
&\sum_{n=k}^N
\frac{(a,q^{-N};q)_n(q^{-n};q)_k}{(c,q;q)_n}
\left(\frac{cq^N}{a}\right)^n=\frac{(c/a;q)_N(a,q^{-N};q)_k}
{(c;q)_N(aq^{1-N}/c;q)_k}.
\end{align*}
Substitution proves \eqref{eq:qgauss-finite}.
\end{proof}

We now extend the terminating parameter $q^{-N}$ to an arbitrary parameter.

\begin{Thm}\label{thm:qgauss}
Let $f$ be holomorphic in $|z|<R$, where $R>1$, and assume
\[
\left|\frac{c}{ab}\right|<R.
\]
Then
\begin{align}
&\sum_{n=0}^{\infty}
\frac{(a,b;q)_n}{(c,q;q)_n}
\left(\frac{c}{ab}\right)^n
\sum_{k=0}^n\frac{(q^{-n};q)_k}{(q;q)_k}f(q^k)q^k \notag\\
&\quad=
\frac{(a,b,c/a,c/b;q)_\infty}
{(q,abq/c,c/ab,c;q)_\infty}
\int_{c/ab}^1
\frac{(tq,abtq/c;q)_\infty}{(at,bt;q)_\infty}f(t)\,d_qt.
\label{eq:qgauss-main}
\end{align}
For fixed $f$, both sides are meromorphic in $a,b,c$ on the region
$|c/ab|<R$, away from their poles. 
\end{Thm}

\begin{proof}
Replace $q^{-N}$ by $1/t$.  The left-hand side of Lemma
\ref{lem:qgauss-finite} is then the specialization at $t=q^N$ of
\begin{equation}\label{eq:L1t}
L_1(t)\coloneqq \sum_{n=0}^{\infty}
\frac{(a,1/t;q)_n}{(c,q;q)_n}
\left(\frac{ct}{a}\right)^n\Delta_n(f).
\end{equation}
Since
\[
t^n(1/t;q)_n=\prod_{j=0}^{n-1}(t-q^j),
\]
Lemma \ref{lem:operators} shows that \eqref{eq:L1t} is holomorphic in a neighborhood of $t=0$.

On the other hand, the right-hand side of \eqref{eq:qgauss-finite} is the specialization at $t=q^N$ of
\begin{align}
R_1(t)
&\coloneqq \frac{(c/a,ct;q)_\infty}{(ct/a,c;q)_\infty}
\frac{(a,1/t;q)_\infty}{(q,aq/ct;q)_\infty}
\int_{ct/a}^1
\frac{(sq,asq/ct;q)_\infty}{(as,s/t;q)_\infty}f(s)\,d_qs.
\label{eq:R1t}
\end{align}
To see this, set
\[
h(s)=\frac{(a;q)_\infty}{(as;q)_\infty}f(s).
\]
Then the product of the second prefactor and the integral in
\eqref{eq:R1t} is
$(q;q)_\infty^{-1}\mathcal J_{c/a,h}(ct/a)$.  Since
$h(q^k)=(a;q)_k f(q^k)$, \eqref{eq:J-terminating}, with regularizing
parameter $c/a$, gives, for all sufficiently large $N$, the following value
at $t=q^N$:
\[
\sum_{k=0}^N
\frac{(a,q^{-N};q)_k}{(q,aq^{1-N}/c;q)_k}f(q^k)q^k.
\]
Moreover,
\[
\frac{(c/a,ct;q)_\infty}{(ct/a,c;q)_\infty}
\bigg|_{t=q^N}
=\frac{(c/a;q)_N}{(c;q)_N}.
\]
Theorem \ref{thm:regularity} also shows that $R_1(t)$ is holomorphic near
$t=0$.  Hence $L_1(q^N)=R_1(q^N)$ for all sufficiently large $N$.  Since
$q^N\to0$, the identity theorem yields $L_1(t)=R_1(t)$ near zero.  Analytic
continuation in $t$ throughout the region $|ct/a|<R$, away from the poles of
$L_1$ and $R_1$,
followed by the substitution $t=1/b$, proves \eqref{eq:qgauss-main} in the
initial parameter domain.  Meromorphic continuation in the remaining
parameters then proves the stated result.  Lemma \ref{lem:operators} also
shows that the outer series converges absolutely, with the terminating inner
sum evaluated first, under $|c/ab|<R$.
\end{proof}

The following basic hypergeometric specialization follows directly from
\eqref{eq:qgauss-main}.

\begin{Cor}\label{cor:qgauss-hypergeometric}
Suppose that
\[
|A_i|\max\left\{1,\left|\frac{c}{ab}\right|\right\}<1
\qquad(1\leq i\leq r).
\]
Then
\begin{align*}
&\sum_{n=0}^{\infty}
\frac{(a,b;q)_n}{(c,q;q)_n}
\left(\frac{c}{ab}\right)^n
{}_{r+1}\phi_r\!\left[
\begin{matrix}q^{-n},A_1,\ldots,A_r\\ B_1,\ldots,B_r\end{matrix};q,q
\right]\\
&\quad=
\frac{(a,b,c/a,c/b,A_1,\ldots,A_r;q)_\infty}
{(q,abq/c,c/ab,c,B_1,\ldots,B_r;q)_\infty}\int_{c/ab}^1
\frac{(tq,abtq/c,B_1t,\ldots,B_rt;q)_\infty}
{(at,bt,A_1t,\ldots,A_rt;q)_\infty}\,d_qt.
\end{align*}
\end{Cor}

\begin{proof}
For generic values of the denominator parameters, define
\[
f(z)=\frac{(A_1,\ldots,A_r,B_1z,\ldots,B_rz;q)_\infty}
{(A_1z,\ldots,A_rz,B_1,\ldots,B_r;q)_\infty}.
\]
The hypothesis permits us to choose
$R>\max\{1,|c/ab|\}$ such that $|A_i|R<1$ for $1\leq i\leq r$.
Hence $f$ is holomorphic in $|z|<R$, and
\[
f(q^k)=\frac{(A_1,\ldots,A_r;q)_k}{(B_1,\ldots,B_r;q)_k}.
\]
The result now follows from Theorem \ref{thm:qgauss}.  The remaining
admissible parameter values follow by meromorphic continuation.
\end{proof}

\section{A double sum of Rogers' \texorpdfstring{${}_6\phi_5$}{6phi5} type}\label{sec:rogers}

We next apply the same continuation procedure to Rogers' ${}_6\phi_5$
summation.  Nonterminating extensions of this summation and related $q$-series
expansion formulas have been studied by Liu \cite{liu6phi5,liu-expansion}.
For later use, we first derive a finite consequence of the strong Bailey lemma
in a normalization adapted to $\Delta_n^{(a)}$, and then use the regularized
Jackson integral to continue the resulting identity.

\begin{Lem}[A finite consequence of the strong Bailey lemma]\label{lem:rogers-finite}
Let $N$ be a nonnegative integer.  Then
\begin{align}
&\sum_{n=0}^N
\frac{(1-aq^{2n})(a,b,c,q^{-N};q)_n}
{(q,aq/b,aq/c,aq^{N+1};q)_n}
\left(\frac{aq^{N+1}}{bc}\right)^n
\Delta_n^{(a)}(f) \notag\\
&\quad=\frac{(a;q)_{N+1}(aq/bc;q)_N}
{(aq/b,aq/c;q)_N}
\sum_{k=0}^N
\frac{(b,c,q^{-N};q)_k}{(bcq^{-N}/a,q;q)_k}f(q^k)q^k.
\label{eq:rogers-finite}
\end{align}
\end{Lem}

\begin{proof}
We apply the strong Bailey lemma in the finite form used by Andrews
\cite[Section~3]{andrews2012}.  Put $\rho_1=b$, $\rho_2=c$, and
\[
\beta_k=\frac{f(q^k)}{(q;q)_k}.
\]
The inverse Bailey relation \cite[Eq.~(3.6)]{andrews2012} gives
\[
\alpha_n=
\frac{(a;q)_n(1-aq^{2n})(-1)^nq^{\binom n2}}
{(q;q)_n(1-a)}\Delta_n^{(a)}(f).
\]
For these parameters, the conjugate sequences in
\cite[Eqs.~(3.3) and (3.4)]{andrews2012} are
\begin{align*}
\gamma_n
&=\frac{(aq/b,aq/c;q)_N}{(aq,aq/bc;q)_N}
\frac{(b,c,q^{-N};q)_n}{(aq/b,aq/c,aq^{N+1};q)_n}
\left(-\frac{aq}{bc}\right)^n
q^{nN-\binom n2},\\
\delta_n
&=\frac{(b,c,q^{-N};q)_n}{(bcq^{-N}/a;q)_n}q^n.
\end{align*}
Substitution into the finite strong Bailey identity
$\sum_{n=0}^N\beta_n\delta_n=
\sum_{n=0}^N\alpha_n\gamma_n$ yields
\begin{align*}
&\sum_{k=0}^N
\frac{(b,c,q^{-N};q)_k}{(bcq^{-N}/a,q;q)_k}f(q^k)q^k\\
&\quad=\frac{(aq/b,aq/c;q)_N}
{(1-a)(aq,aq/bc;q)_N}
\sum_{n=0}^N
\frac{(1-aq^{2n})(a,b,c,q^{-N};q)_n}
{(q,aq/b,aq/c,aq^{N+1};q)_n}
\left(\frac{aq^{N+1}}{bc}\right)^n
\Delta_n^{(a)}(f).
\end{align*}
Rearranging and using $(1-a)(aq;q)_N=(a;q)_{N+1}$ proves
\eqref{eq:rogers-finite}.  The remaining admissible parameter values follow
by meromorphic continuation.
\end{proof}

\begin{Rem}\label{rem:andrews-terminating}
If
\[
f(q^k)=\frac{(D,E;q)_k}{(F,G,H;q)_k},
\qquad FGH=aDEq,
\]
then solving \eqref{eq:rogers-finite} for the finite sum on the right gives
the terminating balanced ${}_5\phi_4$-to-${}_4\phi_3$ expansion of Andrews
\cite[Theorem~4]{andrews2012}.  This is precisely the specialization that
Andrews derives from the same strong Bailey lemma.  Thus the form involving an
arbitrary function is a direct consequence of the strong Bailey lemma; its role here is to supply
the terminating input for the continuation argument below.
\end{Rem}

\begin{Thm}\label{thm:rogers}
Let $f$ be holomorphic in $|z|<R$, where $R>1$, and assume
\[
\left|\frac{aq}{bcd}\right|<R.
\]
Then
\begin{align}
&\sum_{n=0}^{\infty}
\frac{(1-aq^{2n})(a,b,c,d;q)_n}
{(q,aq/b,aq/c,aq/d;q)_n}
\left(\frac{aq}{bcd}\right)^n
\sum_{k=0}^n
\frac{(q^{-n},aq^n;q)_k}{(q;q)_k}f(q^k)q^k \notag\\
&\quad=
\frac{(a,b,c,d,aq/bc,aq/bd,aq/cd;q)_\infty}
{(aq/b,aq/c,aq/d,aq/bcd,bcd/a,q;q)_\infty}\int_{aq/bcd}^1
\frac{(bcdt/a,tq;q)_\infty}{(bt,ct,dt;q)_\infty}f(t)\,d_qt.
\label{eq:rogers-main}
\end{align}
For fixed $f$, both sides are meromorphic in $a,b,c,d$ on the region
$|aq/bcd|<R$, away from their poles.  
\end{Thm}

\begin{proof}
We introduce functions $L_2$ and $R_2$ whose values at $t=q^N$ are the two
sides of Lemma \ref{lem:rogers-finite}:
\begin{align*}
L_2(t)
&\coloneqq \sum_{n=0}^{\infty}
\frac{(1-aq^{2n})(a,b,c,1/t;q)_n}
{(q,aq/b,aq/c,atq;q)_n}
\left(\frac{atq}{bc}\right)^n\Delta_n^{(a)}(f)
\end{align*}
and
\begin{align*}
R_2(t)
&\coloneqq \frac{(a,aq/bc,atq/b,atq/c;q)_\infty}
{(atq,atq/bc,aq/b,aq/c;q)_\infty}
\frac{(b,c,1/t;q)_\infty}{(bc/at,q;q)_\infty}
\int_{atq/bc}^1
\frac{(bcs/at,sq;q)_\infty}{(bs,cs,s/t;q)_\infty}f(s)\,d_qs.
\end{align*}
Indeed, set
\[
h(s)=\frac{(b,c;q)_\infty}{(bs,cs;q)_\infty}f(s).
\]
The product of the second prefactor and the integral defining $R_2(t)$ is
\[
(q;q)_\infty^{-1}\mathcal J_{aq/bc,h}(atq/bc).
\]
Since $h(q^k)=(b,c;q)_k f(q^k)$, \eqref{eq:J-terminating}, with regularizing
parameter $aq/bc$, gives, for all sufficiently large $N$, the following value
at $t=q^N$:
\[
\sum_{k=0}^N
\frac{(b,c,q^{-N};q)_k}{(bcq^{-N}/a,q;q)_k}f(q^k)q^k,
\]
while
\begin{align*}
&\frac{(a,aq/bc,atq/b,atq/c;q)_\infty}
{(atq,atq/bc,aq/b,aq/c;q)_\infty}\bigg|_{t=q^N}=\frac{(a;q)_{N+1}(aq/bc;q)_N}
{(aq/b,aq/c;q)_N}.
\end{align*}
Lemma \ref{lem:operators} and Theorem \ref{thm:regularity} show that $L_2$
and $R_2$ are holomorphic near $t=0$.  Since they agree at $t=q^N$ for all
sufficiently large $N$ and $q^N\to0$, the identity theorem gives
$L_2(t)=R_2(t)$ near zero.  Analytic continuation in $t$ throughout the region
$|atq/bc|<R$, away from the poles of $L_2$ and $R_2$, followed by the substitution $t=1/d$,
proves \eqref{eq:rogers-main} in the initial parameter domain.  Meromorphic
continuation in the remaining parameters then proves the stated result.
Lemma \ref{lem:operators} also shows that the outer series converges
absolutely, with the terminating inner sum evaluated first, under
$|aq/bcd|<R$.
\end{proof}

\begin{Cor}\label{cor:rogers-hypergeometric}
Suppose that
\[
|B_i|\max\left\{1,\left|\frac{aq}{bcd}\right|\right\}<1
\qquad(1\leq i\leq r).
\]
Then
\begin{align*}
&\sum_{n=0}^{\infty}
\frac{(1-aq^{2n})(a,b,c,d;q)_n}
{(q,aq/b,aq/c,aq/d;q)_n}
\left(\frac{aq}{bcd}\right)^n{}_{r+2}\phi_{r+1}\!\left[
\begin{matrix}q^{-n},aq^n,B_1,\ldots,B_r\\ C_1,\ldots,C_{r+1}\end{matrix};q,q
\right]\\
&\quad=\frac{(a,b,c,d,aq/bc,aq/bd,aq/cd,B_1,\ldots,B_r;q)_\infty}
{(aq/b,aq/c,aq/d,aq/bcd,bcd/a,q,C_1,\ldots,C_{r+1};q)_\infty}\\
&\qquad\cdot\int_{aq/bcd}^1
\frac{(bcdt/a,tq,C_1t,\ldots,C_{r+1}t;q)_\infty}
{(bt,ct,dt,B_1t,\ldots,B_rt;q)_\infty}\,d_qt.
\end{align*}
\end{Cor}

\begin{proof}
For generic values of the denominator parameters, define
\[
f(z)=\frac{(B_1,\ldots,B_r,C_1z,\ldots,C_{r+1}z;q)_\infty}
{(B_1z,\ldots,B_rz,C_1,\ldots,C_{r+1};q)_\infty}.
\]
The hypothesis permits us to choose
$R>\max\{1,|aq/bcd|\}$ such that $|B_i|R<1$ for $1\leq i\leq r$.
Hence $f$ is holomorphic in $|z|<R$, and
\[
f(q^k)=\frac{(B_1,\ldots,B_r;q)_k}{(C_1,\ldots,C_{r+1};q)_k}.
\]
The result now follows from Theorem \ref{thm:rogers}.  The remaining
admissible parameter values follow by meromorphic continuation.
\end{proof}

\begin{Rem}
For $r=2$, the inner series in Corollary
\ref{cor:rogers-hypergeometric} is a terminating ${}_4\phi_3$.  Under the
balancing condition
\[
C_1C_2C_3=aB_1B_2q,
\]
the resulting identity recovers the full double-series transformation of
Ismail, Rahman, and Suslov \cite[Theorem~1.1]{irs}, rather than merely a
specialization of it.  To see this, denote the parameters in their theorem by
$a,b_0,c_0,d_0,e_0,f_0,g,h$.  Sears' transformation for a terminating
balanced ${}_4\phi_3$ \cite[Eq.~(III.15)]{gr} gives
\begin{align*}
&{}_4\phi_3\!\left[
\begin{matrix}q^{-n},aq^n,g,h\\ b_0,c_0,aqgh/b_0c_0\end{matrix};q,q
\right]\\
&\quad=\frac{(aq/b_0,aq/c_0;q)_n}{(b_0,c_0;q)_n}
\left(\frac{b_0c_0}{aq}\right)^n
{}_4\phi_3\!\left[
\begin{matrix}
q^{-n},aq^n,aqg/b_0c_0,aqh/b_0c_0\\
aqgh/b_0c_0,aq/b_0,aq/c_0
\end{matrix};q,q
\right].
\end{align*}
Apply this transformation to the inner series on the left-hand side of
\cite[Theorem~1.1]{irs}.  The prefactor in Sears' transformation cancels the factors involving
$b_0,c_0,aq/b_0$, and $aq/c_0$ in the outer summand.  Using
\[
\frac{(q\sqrt a,-q\sqrt a;q)_n}{(\sqrt a,-\sqrt a;q)_n}
=\frac{1-aq^{2n}}{1-a},
\]
we see that the transformed left-hand side, multiplied by $1-a$, is exactly
the left-hand side of Corollary \ref{cor:rogers-hypergeometric} with
\[
(b,c,d)=(d_0,e_0,f_0),
\]
and
\[
(B_1,B_2)=\left(\frac{aqg}{b_0c_0},\frac{aqh}{b_0c_0}\right),\qquad
(C_1,C_2,C_3)=\left(\frac{aqgh}{b_0c_0},
\frac{aq}{b_0},\frac{aq}{c_0}\right).
\]
The balancing condition is automatic.  With the same substitution, expanding
the Jackson integral along its two $q$-lattices gives the two ${}_5\phi_4$
series on the right-hand side of \cite[Theorem~1.1]{irs}; after cancellation
of the product factors, the prefactors are $(1-a)$ times the corresponding
prefactors in that theorem.  Thus Corollary
\ref{cor:rogers-hypergeometric}, followed by the inverse of Sears' transformation,
recovers \cite[Theorem~1.1]{irs} without any additional specialization of its
parameters.  Together
with Remark
\ref{rem:andrews-terminating}, this shows that Theorem \ref{thm:rogers}
should be regarded as a functional and integral extension of known balanced
expansions, rather than as a new terminating ${}_5\phi_4$ transformation.
\end{Rem}

\end{document}